\documentclass[3p,12pt]{elsarticle}

\makeatletter						
\def\ps@pprintTitle{%
  \let\@oddhead\@empty
  \let\@evenhead\@empty
  \let\@oddfoot\@empty
  \let\@evenfoot\@oddfoot
}
\makeatother

\usepackage{amssymb}
\usepackage{amsmath}
\usepackage{amsthm}
\usepackage{enumerate}
\usepackage{amsfonts}
\usepackage{graphicx}
\usepackage{multicol}
\usepackage{bussproofs}
\usepackage{setspace}

\usepackage{tikz}
\usetikzlibrary{shapes,arrows,calc,positioning}
\tikzstyle{bull}=[circle,draw=black,fill=black!80]
\tikzstyle{holl}=[circle,draw=black]
\usetikzlibrary{matrix}

\usepackage[hidelinks]{hyperref}
\hypersetup{colorlinks=true,citecolor=blue,linkcolor=blue}

\allowdisplaybreaks  					

\newtheoremstyle{myplain}
  {\topsep}   
  {\topsep}   
  {\itshape\onehalfspacing}  
  {0pt}       
  {\bfseries} 
  {.}         
  {5pt plus 1pt minus 1pt} 
  {}          

\newtheoremstyle{mydefinition}
  {\topsep}   
  {\topsep}   
  {\normalfont\onehalfspacing}  
  {0pt}       
  {\bfseries} 
  {.}         
  {5pt plus 1pt minus 1pt} 
  {}          
  
\theoremstyle{myplain}
\newtheorem{theorem}{Theorem}[section]
\newtheorem{proposition}[theorem]{Proposition}
\newtheorem{lemma}[theorem]{Lemma}
\newtheorem{corollary}[theorem]{Corollary}

\theoremstyle{mydefinition}
\newtheorem{definition}[theorem]{Definition}

\newtheorem{remark}[theorem]{Remark}

\numberwithin{equation}{section}

\usepackage[hidelinks]{hyperref}
\hypersetup{colorlinks=true,citecolor=blue,linkcolor=blue}

\newcommand{\C}{\mathcal{C}}

\newcommand{\Hom}{\mathrm{Hom}}
\newcommand{\Fm}{\mathrm{Fm}}
\newcommand{\var}{\mathrm{Var}}
\newcommand{\lan}{\mathcal{L}}
\renewcommand{\S}{\mathcal{S}}
\newcommand{\Th}{\mathcal{T}h}
\newcommand{\Alg}{\mathrm{Alg}}

\newcommand{\V}{\mathbb{V}}

\newcommand{\N}{\mathbb{N}}

\newcommand{\Fi}{\mathcal{F}i}
\newcommand{\Fig}{\mathrm{Fig}}

\newcommand{\Ev}{\mathcal{E}}
\newcommand{\Od}{\mathcal{O}}

\begin{document}

\title{Algebraic logic for the negation fragment of classical logic}

\author{Luciano J. Gonz\'alez\tnoteref{t1}}
\address{CONICET (Argentina) -- Universidad Nacional de La Pampa. Facultad de Ciencias Exactas y Naturales. Santa Rosa, Argentina.}
\ead{lucianogonzalez@exactas.unlpam.edu.ar}

\begin{abstract}
The general aim of this article is to study the negation fragment of classical logic within the framework of contemporary (Abstract) Algebraic Logic. More precisely, we shall find the three classes of algebras that are canonically associated with a logic in Algebraic Logic, that is, we find the classes $\Alg^*$, $\Alg$ and the intrinsic variety of the negation fragment of classical logic. In order to achieve this, firstly we propose a Hilbert-style axiomatization for this fragment. Then, we characterize the reduced matrix models and the full generalized matrix models of this logic. Also, we classify the negation fragment in the Leibniz and Frege hierarchies.
\end{abstract}


\maketitle

\section{Introduction}

It is clear that the negation fragment of classical propositional logic (from now on CPL) is a very inexpressive logic from the syntactic point of view. One can only consider up to equivalence a proposition $p$ and its negation $\neg p$, and no other compound sentential can be built up. However, this fragment has several algebra-based semantics (in the different senses of this term) that are not so trivial. We intend to describe these algebra-based semantics of the negation fragment of CPL from the point of view of algebraic logic. 

In the framework of (abstract) algebraic logic, there are essentially three classes of algebras associated with a propositional logic. These classes are obtained from different procedures, which intend to be a kind of generalization or abstraction of the Tarski-Lindenbaum method applied to classical (intuitionistic) propositional logic. The class $\Alg^*(\S)$  is the class of algebras that is canonically associated with a logic $\S$ according to the theory of logical matrices. More precisely, $\Alg^*(\S)$ is the class of the algebraic reducts of the reduced matrix models of $\S$. The intrinsic variety of a logic $\S$ (denoted by $\V(\S)$) is defined as the variety generated by a quotient algebra on the formula algebra. And the class $\Alg(\S)$ is determined from the reduced generalized matrix models, that is, $\Alg(\S)$ is the class of algebraic reducts of the reduced g-models of $\S$. We refer the reader to \cite{Fo16,FoJa09,FoJaPi03} for the specific definitions. In general, for any logic $\S$, we always have that $\Alg^*(\S)\subseteq\Alg(\S)\subseteq\V(\S)$. In many cases, the three classes coincide, and in many others cases the inclusions are proper. From the point of view of  algebraic logic, the class of algebras which is more representative for a logic $\S$ or which can be considered as the \emph{algebraic counterpart} of $\S$ is the class $\Alg(\S)$. We refer the reader to \cite{Fo16,FoJa09,FoJaPi03} for a deeper explanation of the aims and goals of algebraic logic and the corresponding algebra-based semantics. We intend to obtain the classes $\Alg^*$, the intrinsic variety and the class $\Alg$ for the negation fragment of CPL.

The article is organized as follows. Section \ref{sec:significance} is about the significance of the negation fragment of CPL, and we try to justify why it is important to study algebraically this fragment. In Section \ref{sec:2}, we introduce the very basics to start with the work. Throughout the paper, we will introduce the needed concepts and results of algebraic logic. We assume that the reader is familiar with the very basics of algebraic logic, for instance, with the notions of propositional logic (or sentential logic), logical filter, a theory of a logic, logical matrices, generalized matrices, etc. We shall provide all those notions that might be less usual for the reader. We refer the reader to \cite{Fo16,FoJa09} for further information on algebraic logic. Section \ref{sec:3} presents a Hilbert-style system for the negation fragment of CPL and shows a completeness theorem. In Section \ref{sec:4}, we characterize the Leibniz-reduced matrix models and we describe the class $\Alg^*$ of the negation fragment. Section \ref{sec:5} is devoted to obtaining the intrinsic variety of the negation fragment. In Section \ref{sec:6}, we characterize the full g-models of the negation fragment. In order to obtain this, we describe the logical filters generated by a set on an arbitrary algebra and present several properties of these logical filters. Then, we find the class $\Alg$ of the negation fragment. Finally, in Section \ref{sec:7}, we classify the negation fragment in the Leibniz and Frege hierarchies.

\section{Significance}\label{sec:significance}

As noted in the introduction, the negation fragment of CPL is a very simple logic since its algebraic language has only one unary connective. In this section, we want to justify and convince  the reader about the importance of having an algebraic description of this fragment. So, we ask ourselves, what is the algebraic counterpart, under the Algebraic Logic (\cite{FoJaPi03,FoJa09,Fo16}) point of view, of the negation fragment of CPL? This question was addressed in the literature for the others  fragments of CPL, see Table \ref{tab:fragments}. It is also important to answer the above question for the negation fragment beyond its simplicity.

In order to study propositional logics is often important to have an axiomatization, for instance, a Hilbert-style or Gentzen-style axiomatization. By \cite{Ra81} it is known that the negation fragment of CPL has a finite Hilbert-style axiomatization. But there it isn't explicitly presented. However, it is known from the folklore that the Hilbert-style rules $x,\neg x\vdash y$, $x\vdash\neg\neg x$ and $\neg\neg x\vdash x$ are an axiomatization of the negation fragment of CPL. As far as we know, there isn't in the literature an argumentation of this. Here we present one.

As mentioned in the introduction, for the negation fragment of CPL, we described the three classes of algebras that are naturally associated with a propositional logic in Algebraic Logic. In spite of the syntactical simplicity of the negation fragment of CPL, we show that these three classes of algebras are different and they are not so simple. In particular, we characterize the class $\Alg$ for this fragment. We notice that this class of algebras was recently obtained in \cite{JaMo21} using the concept of Suszko-reduced matrix. In this paper, we follow an alternative path to describe the class $\Alg$. We use the Tarski-reduced full g-models. On the one hand, a logical matrix is a pair $\langle A,F\rangle$ where $A$ is an algebra (over a corresponding algebraic language) and $F\subseteq A$. On the other hand, a generalized matrix (g-matrix) is a pair $\langle A,\C\rangle$ where $A$ is an algebra and $\C$ is a closure system on $A$. Both structures serve to establish algebra-based semantics for propositional logics (see Section \ref{sec:4} and \ref{sec:6}). These two algebra-based semantics have their differences, and regarding the more general difference between them, we want to quote Font and Jansana (quoted from \cite{FoJa09}): 

\begin{quote}
``Since an abstract logic can be viewed as a``bundle'' or family of matrices, one might think that the new models are essentially equivalent to the old ones; but we believe, after an
overall appreciation of the work done in this area, that it is precisely the treatment
of an abstract logic as a single object what gives rise to a useful--and beautiful--mathematical theory, able to explain the connections, not only at the logical level
but at the metalogical level, between a sentential logic and the particular class of
models we associate with it, namely the class of its full models.''
\end{quote}

\noindent Moreover, to justify why characterize the Tarski congruence and the Tarski-reduced full g-models rather than the Suszko congruence and the Suszko-reduced models, respectively, we quote Font et al. (quoted from \cite[p. 73]{FoJaPi03}):

\begin{quote}
``The Suszko congruences appear as particular cases of the Tarski congruence [\dots].''
\end{quote}

In this article, we obtain a description of the full g-models for the negation fragment of CPL, showing that they are not simple at all. In order to convince the reader why to characterize the full g-models of the negation fragment, which seem to be more complex than the negation fragment itself,  we quote Font and Jansana in \cite[p. 3]{FoJa09}:

\begin{quote}
``We associate with each sentential logic $\S$ a class of abstract logics called \emph{the full models of} $\S$ [\dots] with the conviction that (some of) the interesting metalogical properties of
the sentential logic are precisely those shared by its full models. [\dots] And we claim that \emph{the notion of full model is a ``right'' notion of model of a sentential logic} [\dots].''
\end{quote}

Beyond the scope of this article, we want to mention that there is also a great interest in studying negation from the philosophical, linguistics, artificial intelligence and logic programming point of view. We refer the reader to \cite{GaWa99} where there is a compendium of articles studying negation from different perspectives addressed to the question: What is negation? For instance, in \cite{Du99} Dunn discusses several properties that a negation can have: $\varphi\vdash\psi$ only if $\neg\psi\vdash\neg\varphi$ (\emph{contraposition)}; $\varphi\vdash-\neg\psi$ (\emph{Galois double negation}); $\varphi\vdash\neg\neg\varphi$ (\emph{constructive double negation}); $\varphi\vdash\neg\psi$ only if $\psi\vdash\neg\varphi$ (\emph{constructive contraposition}); $\varphi\vdash\psi$ and $\varphi\vdash\neg\psi$ only if $\varphi\vdash\chi$ (\emph{absurdity}); $\neg\neg\varphi\vdash\varphi$ (\emph{classical double negation}); $\neg\varphi\vdash\psi$ only if $\neg\psi\vdash\varphi$ (\emph{classical contraposition}). These properties are considered in different contexts. In \cite{Du99} Dunn studies several connections between different treatments of the semantics of negation in non-classical logics: the Kripke definition of negation for intuitionistic logic, the Goldblatt's semantics for negation in orthologic,  the definition of De Morgan negation in relevant logic, the four-valued semantics of De Morgan negation, and the star semantics. Dunn provides a detailed correspondence-theoretic classification of various notions of negation in terms of properties of a binary relation interpreted as incompatibility.

\section{The $\{\neg\}$-fragment of classical propositional logic (CPL)}\label{sec:2}

Throughout what follows, we establish the following simple conventions. Given a function $f\colon X\to Y$ and $A\cup\{x\}\subseteq X$, we denote $fx$ instead of $f(x)$ and $fA=\{fa: a\in A\}$.

Let $\lan=\{\neg\}$ be an algebraic language of type (1) and let $\Fm$ be the algebra of formulas over the language $\lan$ and generated by a countably infinite set $\var$. Unless otherwise stated, all the algebras considered in the paper are defined over the algebraic language $\lan$. Let us denote by $\S_N=\langle\Fm,\vdash_N\rangle$ the $\{\neg\}$-fragment of CPL, where $\vdash_N$ is the corresponding consequence relation. Let ${\bf 2}_\neg=\langle\{0,1\},\neg\rangle$ be the $\{\neg\}$-reduct of the two-element Boolean algebra ${\bf 2}_B$. Then, it is clear that for all $\Gamma\cup\{\varphi\}\subseteq\Fm$,
\[
\Gamma\vdash_N\varphi \quad \iff \quad \forall h\in\Hom(\Fm,{\bf 2}_\neg) (h\Gamma\subseteq\{1\}\implies h\varphi=1).
\]

Let $\N$ be the set of all positive integers and $\N_0=\N\cup\{0\}$. Let $x\in\var$. For each $n\in\N_0$, we define recursively $\neg^nx$ as usual:
\[
\begin{cases}
\neg^0x &= x\\
\neg^nx &=\neg \neg^{n-1}x \quad \forall n\geq 1.
\end{cases}
\]

\begin{lemma}\label{lem: abs free alg}
For every $\alpha\in\Fm$, there is $x\in\var$ and $n\in\mathbb{N}_0$ such that $\alpha=\neg^nx$.
\end{lemma}

\begin{proof}
It is straightforward because $\Fm$ is the absolutely free algebra of type $\{\neg\}$ over $\var$.
\end{proof}

\section{Hilbert-style axiomatization for the $\{\neg\}$-fragment of CPL}\label{sec:3}

In \cite[Theorem 1]{Ra81}, it is claimed that every two-valued logic (a logic defined as usual by a two-valued matrix $M=\langle A,\{1\}\rangle$, that is, $A=\{0,1\}$ is a two-element algebra) has a finite Hilbert-style axiomatization. Hence,  it follows  that $\S_N$ has a finite Hilbert-style axiomatization. In \cite[p.~322]{Ra81}, it is mentioned that a Hilbert-style axiomatization for $\S_N$ is easily established, but it isn't explicitly presented. It is part of the folklore that the negation fragment of CPL is axiomatized by the following rules: $x,\neg x\vdash y$, $x\vdash\neg\neg x$ and $\neg\neg x\vdash x$. We give a proof of it for the lack of proper reference.

\begin{definition}
Let $\S_\neg=\langle\Fm,\vdash_\neg\rangle$ be the propositional logic defined, as usual, by the following Hilbert-style system:
\begin{enumerate}[(R1)]
	\item $x,\neg x\vdash y$
	\item $x\vdash\neg\neg x$
	\item $\neg\neg x\vdash x$.
\end{enumerate}
\end{definition}

Then, our goal is to show that the logics $\S_N$ and $\S_\neg$ coincide. To this end, we need some auxiliary results. Given a propositional logic $\S=\langle\Fm,\vdash\rangle$, a subset of formulas $\Gamma$ is said to be \emph{inconsistent} if $\Gamma\vdash\alpha$ for all $\alpha\in\Fm$. Otherwise $\Gamma$ is said to be \emph{consistent}.

\begin{proposition}\label{prop: super finitary}
Let $\Gamma\subseteq\Fm$ be consistent. If $\Gamma\vdash_\neg\alpha$, then there is $\gamma\in\Gamma$ such that $\gamma\vdash_\neg\alpha$.
\end{proposition}

\begin{proof}
\onehalfspacing
Suppose that $\Gamma\vdash_\neg\alpha$. We proceed by induction on the length of the proofs from $\Gamma$. That is, we prove that for all $n\in\N$, if $\alpha_1,\dots,\alpha_n$ is a proof from $\Gamma$, then there is $\gamma\in\Gamma$ such that $\gamma\vdash_\neg\alpha_n$.

If $n=1$, then $\alpha_1$ is a proof from $\Gamma$. Then $\alpha_1\in\Gamma$. Hence, there is $\gamma:=\alpha_1\in\Gamma$ such that $\gamma\vdash_\neg\alpha_1$. Now suppose that for each proof $\alpha_1,\dots,\alpha_k$ from $\Gamma$ of length $k<n$, there is $\gamma\in\Gamma$ such that $\gamma\vdash_\neg\alpha_k$. Let $\alpha_1,\dots,\alpha_n$ be a proof from $\Gamma$. Then, $\alpha_n\in\Gamma$ or there is $i<n$ such that $\alpha_n$ is derivable from the rules (R2) or (R3) and $\alpha_i$ (notice that $\alpha_n$ cannot be derivable from the rule (R1) because $\Gamma$ is consistent). Thus, in any case, $\alpha_i\vdash_\neg\alpha_n$. By inductive hypothesis, there is $\gamma\in\Gamma$ such that $\gamma\vdash_\neg\alpha_i$. Then $\gamma\vdash_\neg\alpha_n$.
\end{proof}

\begin{proposition}\label{prop:property of vdash_N}
Let $\alpha,\beta\in\Fm$. Then, $\alpha\vdash_\neg\beta$ if and only if there is $x\in\var$ and $n,k\in\N_0$ such that $\alpha=\neg^nx$, $\beta=\neg^kx$ and $n\equiv_2k$.
\end{proposition}

\begin{proof}
\onehalfspacing
$(\Rightarrow)$ Assume that $\alpha\vdash_\neg\beta$. We proceed by induction on the length of the proofs from $\alpha$. That is, we prove that for all $m\in\N$, if $\alpha_1,\dots,\alpha_m$ is a proof from $\alpha$, then there is $x\in\var$ and $n,k\in\N_0$ such that $\alpha=\neg^nx$, $\alpha_m=\neg^kx$ and $n\equiv_2k$. If $m=1$, then $\alpha=\alpha_1$. From Lemma \ref{lem: abs free alg}, we have $\alpha=\neg^nx=\alpha_1$ for some $x\in\var$ and $n\in\N_0$. Now suppose that it holds for all $i<m$. Let $\alpha_1,\dots,\alpha_m$ be a proof from $\alpha$. Then, $\alpha_m=\alpha$ or $\alpha_m$ is derivable from  $\alpha_i$ with $i<m$ by an application of (R2) or (R3) (it cannot be derivable by (R1), otherwise there are $\alpha_i=\beta$ and $\alpha_j=\neg\beta$ with $i,j<m$. Then by inductive hypothesis there is a variable $x$ and $n,m,k\in\N_0$ such that $\alpha=\neg^nx$, $\alpha_i=\beta=\neg^m x$, $\alpha_j=\neg\beta=\neg^kx$, $n\equiv_2m$ and $n\equiv_2k$. Hence, $m\equiv_2k$ and $m+1=k$, a contradiction). It is straightforward when $\alpha_m=\alpha$. Suppose that $\alpha_m$ is derivable from  $\alpha_i$ with $i<m$ by an application of (R2) or (R3). By inductive hypothesis, there is $x\in\var$ and $n,k\in\N_0$ such that $\alpha=\neg^nx$, $\alpha_i=\neg^kx$ and $n\equiv_2k$. On the one hand, if $\alpha_m$ is derivable from (R2), then $\alpha_m=\neg\neg\alpha_i$. Hence $\alpha_m=\neg\neg\alpha_i=\neg^{k+2}x$ and $n\equiv_2k+2$. On the other hand, if $\alpha_m$ is derivable from (R3), then $\alpha_i=\neg\neg\alpha_m$. Since $\neg\neg\alpha_m=\alpha_i=\neg^kx$, we have that $k\geq 2$. Then $\alpha_m=\neg^{k-2}x$ and $k-2\geq 0$. Hence, $\alpha=\neg^nx$, $\alpha_m=\neg^{k-2}x$ and $n\equiv_2k-2$.

$(\Leftarrow)$ Suppose that there are $x\in\var$ and $n,k\in\N_0$ such that $\alpha=\neg^nx$, $\beta=\neg^kx$ and $n\equiv_2k$. Suppose that $n\leq k$. Thus $0\leq k-n=2q$, for some $q\in\N_0$. Then, we have $\alpha=\neg^nx\vdash_\neg\neg^{n+2}x\vdash_\neg\dots\vdash_\neg\neg^{n+2q}x=\neg^kx=\beta$ (this can be proved by induction). If $n>k$, then $n-2q=k$. Hence $\alpha=\neg^nx\vdash_\neg\neg^{n-2}x\vdash_\neg\dots\vdash_\neg\neg^{n-2q}x=\beta$.
\end{proof}

\begin{remark}
Notice that for all $\alpha,\beta\in\Fm$, $\alpha\vdash_\neg\beta$ if and only if $\beta\vdash_\neg\alpha$. In other words, $\alpha\vdash_\neg\beta$ if and only if $\alpha\dashv\vdash_\neg\beta$.
\end{remark}

\begin{corollary}\label{coro:property}
Let $\Gamma\cup\{\alpha,\beta\}\subseteq\Fm$.
\begin{enumerate}
	\item $\alpha\vdash_\neg\beta \implies \neg\alpha\nvdash_\neg\beta$.
	\item $\alpha\vdash_\neg\beta \iff \neg\beta\vdash_\neg\neg\alpha$.
	\item $\Gamma,\alpha\vdash_\neg\beta$ and $\Gamma,\neg\alpha\vdash_\neg\beta$ $\implies$ $\Gamma\vdash_\neg\beta$.
\end{enumerate}
\end{corollary}

\begin{proof}
\onehalfspacing
(1) Suppose $\alpha\vdash_\neg\beta$. Thus, there are $x\in\var$ and $n,k\in\N_0$ such that $\alpha=\neg^nx$, $\beta=\neg^kx$ and $n\equiv_2 k$. Then $\neg\alpha=\neg^{n+1}x$ and $n+1\not\equiv_2k$. Hence $\neg\alpha\nvdash_\neg\beta$.

(2) $(\Rightarrow)$ Suppose that $\alpha\vdash_\neg\beta$. Then, there are $x\in\var$ and $n,k\in\N_0$ such that $\alpha=\neg^nx$, $\beta=\neg^kx$ and $n\equiv_2k$. Thus $\neg\alpha=\neg^{n+1}x$, $\neg\beta=\neg^{k+1}x$ and $n+1\equiv_2k+1$. Hence $\neg\beta\vdash_\neg\neg\alpha$. $(\Leftarrow)$ It is straightforward by the above and by rules (R2) and (R3).

(3) Assume that $\Gamma,\alpha\vdash_\neg\beta$ and $\Gamma,\neg\alpha\vdash_\neg\beta$. Suppose that $\Gamma\nvdash_\neg\beta$. Then, by Proposition \ref{prop: super finitary}, it follows that $\alpha\vdash_\neg\beta$ and $\neg\alpha\vdash_\neg\beta$. This is a contradiction by property (1). Hence $\Gamma\vdash_\neg\beta$.
\end{proof}

\begin{proposition}\label{prop:property of max theory}
Let $\alpha\in\Fm$ and let $\Delta$ be a maximal theory relative to $\alpha$ of the logic $\S_\neg$ (that is, $\Delta$ is a maximal theory among all the consistent theories not containing $\alpha$). Then, for all $\beta\in\Fm$, $\beta\in\Delta$ iff $\neg\beta\notin\Delta$.
\end{proposition}

\begin{proof}
\onehalfspacing
Let $\beta\in\Fm$.

$(\Rightarrow)$ Assume $\beta\in\Delta$. Since $\Delta$ is consistent ($\alpha\notin\Delta$), it follows by rule (R1) that $\neg\beta\notin\Delta$.

$(\Leftarrow)$ Suppose that $\neg\beta\notin\Delta$. We suppose, towards a contradiction, that $\beta\notin\Delta$. Let $\Gamma_\beta$ and $\Gamma_{\neg\beta}$ be the theories generated by $\Delta\cup\{\beta\}$ and $\Delta\cup\{\neg\beta\}$, respectively. Since $\Delta\subset\Gamma_\beta$ and $\Delta\subset\Gamma_{\neg\beta}$, it follows by the maximality of $\Delta$ that $\alpha\in\Gamma_\beta\cap\Gamma_{\neg\beta}$. Then
\[
\Delta,\beta\vdash\alpha \qquad \text{and} \qquad \Delta,\neg\beta\vdash\alpha.
\]
Hence, by (3) of Corollary \ref{coro:property}, we obtain that $\Delta\vdash\alpha$, which is a contradiction. Therefore, $\beta\in\Delta$.
\end{proof}

Now we are ready to show that the rules (R1)--(R3) are an axiomatization for $\S_N$.

\begin{theorem}
The Hilbert calculus formed by the rules (R1), (R2) and (R3) is an axiomatization for the $\{\neg\}$-fragment of CPL.
\end{theorem}

\begin{proof}
\onehalfspacing
We need to show that $\vdash_N{=}\vdash_\neg$. Recall that $\vdash_N$ is defined by the matrix $\langle{\bf 2}_\neg,\{1\}\rangle$.

$(\vdash_\neg{\subseteq}\vdash_N)$ (Soundness) This is a routine argument.

$(\vdash_N{\subseteq}\vdash_\neg)$ Suppose that $\Gamma\nvdash_\neg\alpha$. By \cite[Lem. 1.43]{Fo16}, there is a theory $\Delta$ such that $\Gamma\subseteq\Delta$, $\alpha\notin\Delta$, and $\Delta$ is maximal relative to $\alpha$. We define $v\colon\Fm\to{\bf 2}_\neg$ as follows: for all $\varphi\in\Fm$, $v\varphi=1 \iff \varphi\in\Delta$. By Proposition \ref{prop:property of max theory}, we obtain that $v$ is a homomorphism such that $v\Gamma\subseteq\{1\}$ and $v\alpha=0$. Hence $\Gamma\nvdash_N\alpha$.
\end{proof}

\section{Reduced models and the class $\Alg^*(\S_N)$}\label{sec:4}

In this section, we characterize the reduced matrix models of the logic $\S_N$ and we describe the class $\Alg^*(\S_N)$. We recall some needed notions.

Recall that a \emph{logical matrix} (\emph{matrix} for short) is a pair $\langle A,F\rangle$ where $A$ is an algebra  and $F\subseteq A$. The \emph{Leibniz congruence}, denoted by $\Omega^AF$, of a matrix $\langle A,F\rangle$ can be defined as follows (see \cite[Theo. 4.23]{Fo16}): for all $a,b\in A$,
\begin{equation}\label{equa:def of Leibniz congruence}
\begin{split}
(a,b)\in\Omega^AF \iff &\text{for all } \delta(x,\overrightarrow{z})\in\Fm \text{ and all }\overrightarrow{c}\in\overrightarrow{A},\\
&[\delta^A(a,\overrightarrow{c} )\in F\iff\delta^A(b,\overrightarrow{c})\in F].
\end{split}
\end{equation}
A matrix $\langle A,F\rangle$ is said to be \emph{reduced} when $\Omega^AF=\mathrm{Id}_A$. Recall that a matrix $\langle A,F\rangle$ is a \emph{model} of a logic $\S$ when for all $\Gamma\cup\{\alpha\}\subseteq\Fm$,
\begin{equation}\label{equa:def of model}
\Gamma\vdash_\S\alpha \implies \forall h\in\Hom(\Fm,A)(h\Gamma\subseteq F\implies h\alpha\in F).
\end{equation}
Then, the class $\Alg^*(\S)$ is defined as follows;
\[
\Alg^*(\S)=\{A: \text{ there is some } F\subseteq A \text{ such that } \langle A,F\rangle \text{ is a reduced model of } \S\}.
\]
Recall also that a subset $F$ of an algebra $A$ is called an $\S$-\emph{filter} of $A$ if condition \eqref{equa:def of model} is satisfied. That is, a subset $F\subseteq A$ is an $\S$-filter of $A$ if and only if the matrix $\langle A,F\rangle$ is a model of $\S$.

By \eqref{equa:def of Leibniz congruence} and taking into account that every formula $\alpha\in\Fm$ is of the form $\alpha=\neg^nx$ for some $x\in\var$ and $n\in\N_0$, we obtain the following characterization of the Leibniz congruences.

\begin{proposition}\label{prop:Leibniz cong}
For every algebra $A$ and every $F\subseteq A$,
\[
(a,b)\in\Omega^AF \iff \forall n\in\N_0(\neg^na\in F \iff \neg^nb\in F)
\]
for all $a,b\in A$.
\end{proposition}

\begin{proposition}\label{prop: charact filter}
Let $A$ be an algebra and $F\subseteq A$. Then, $F$ is an $\S_N$-filter if and only if the following conditions hold:
\begin{enumerate}[\normalfont (1)]
	\item $a,\neg a\in F$ $\implies$ $F=A$;
	\item $a\in F$ $\iff$ $\neg\neg a\in F$.
\end{enumerate}
\end{proposition}

\begin{proof}
It is straightforward by rules (R1)--(R3).
\end{proof}

We are ready to characterize the reduced models of $\S_N$.

\begin{theorem}
Let $\langle A,F\rangle$ be a matrix. Then, $\langle A,F\rangle$ is a reduced model of $\S_N$ if and only if the following conditions hold:
\begin{enumerate}[\normalfont (1)]
	\item $A\models x\approx\neg\neg x$,
	\item $F=\{a_0\}$ for some $a_0\in A$ such that $a_0\neq \neg a_0$, and
	\item $2\leq |A|\leq 3$.
\end{enumerate}
\end{theorem}

\begin{proof}
\onehalfspacing
$(\Rightarrow)$ Assume that $\langle A,F\rangle$ is a reduced model of $\S_N$. Thus $F$ is an $\S_N$-filter and $\Omega^AF=\mathrm{Id}_A$.

\noindent (1) Let $a\in A$.  From Proposition \ref{prop: charact filter}, we obtain
\[
\forall n\in \N_0(\neg^na\in F \iff \neg^n(\neg\neg a)\in F).
\]
Then, $(a,\neg\neg a)\in\Omega^AF$. Hence $a=\neg\neg a$.

\noindent (2) We are assumed that the algebra $A$ is not trivial, that is, $|A|\geq 2$. Then, $F\neq\emptyset$ (otherwise $\Omega^AF=A\times A\neq\Delta_A$). Let $a,b\in F$. By (1), we have that  $a=\neg^{2k}a$ and $b=\neg^{2k}b$ for all $k\in\N_0$. Thus, $\neg^{2k}a,\neg^{2k}b\in F$, for all $k\in\N_0$. On the other hand, since $a,b\in F$ and $F$ is proper (otherwise $\langle A,F\rangle$ is not reduced), it follows by Proposition \ref{prop: charact filter} that $\neg a,\neg b\notin F$. Then, $\neg^{2k+1}a,\neg^{2k+1}b\notin F$, for all $k\in\N_0$. Hence, we have that $\forall n\in\N_0(\neg^na\in F\iff \neg^nb\in F)$. Thus, $(a,b)\in\Omega^AF$. Then, $a=b$. Therefore, $F=\{a_0\}$ for some $a_0\in A$. Moreover $a_0\neq \neg a_0$. Otherwise, $a_0=\neg a_0\in F$, and by Proposition \ref{prop: charact filter} we have $F=A$, which is a contradiction because $A$ is not trivial.

\noindent (3) Since $A$ is not trivial, we have $|A|\geq 2$. By (2), we have that $F=\{a_0\}$ and $a_0\neq \neg a_0$. Let $a,b\in A$ be such that $a,b\notin\{a_0,\neg a_0\}$. Then, $a,b,\neg a,\neg b\notin F=\{a_0\}$. Thus, it holds that $a\in F \iff b\in F$ and $\neg a\in F \iff \neg b\in F$. Now by (1), it follows that $\forall n\in\N_0(\neg^na\in F\iff \neg^nb\in F)$. Then $(a,b)\in\Omega^AF$. Hence, $a=b$. Therefore, $|A|\leq 3$.

$(\Leftarrow)$ Let $\langle A, F\rangle$ be a matrix such that satisfies (1)--(3). By (1) and (2), it follows that $F=\{a_0\}$ is an $\S_N$-filter, that is, $\langle A, F\rangle$ is a model of $\S_N$. Let us see that $\langle A,F\rangle$ is reduced. Let $a,b\in A$ and assume that $(a,b)\in\Omega_AF$. Thus, we have $a\in F \iff b\in F$ and $\neg a\in F \iff \neg b\in F$. If $a\in F$ or $\neg a\in F$, then $a=b$. Similarly, if $b\in F$ or $\neg b\in F$, then $a=b$. Now suppose that $a,b,\neg a,\neg b\notin F=\{a_0\}$. Thus $a,b\notin\{a_0,\neg a_0\}$. Since $|A|\leq 3$, we obtain that $a=b$. Hence $\Omega^AF=\Delta_A$. Therefore, $\langle A,F\rangle$ is a reduced model of $\S_N$.
\end{proof}

Therefore, we have that all non-trivial reduced matrix models $\langle A,F\rangle$ are of the form, up to isomorphism:
\begin{enumerate}
	\item $A_1=\{a\}$ and $F=\emptyset$;
	\item $A_2=\{a,b\}$ such that $\neg a=b$ and $\neg b=a$, with $F=\{a\}$; or \label{alg A2}
	\item $A_3=\{a,b,c\}$ such that $\neg a=b$, $\neg b=a$ and $\neg c=c$, with $F=\{a\}$. \label{alg A3}
\end{enumerate}

\begin{corollary}\label{coro:Alg*}
$\Alg^*(\S_N)=\mathbb{I}\left(\{A_1,A_2,A_3\}\right)$.
\end{corollary}

\section{The intrinsic variety of $\S_N$}\label{sec:5}

If one insists on having a variety associated with a logic $\S$, the \emph{intrinsic variety} of $\S$ might be the adequate choice. The significance of the intrinsic variety as an algebra-based semantics for a logic is in general weak, because no general theory asserts that the algebraic counterpart of a logic should be a variety. In order to define the intrinsic variety and for what follows in the article, we recall some needed concepts.

Recall that a \emph{generalized matrix} (\emph{g-matrix}) is a pair $\langle A,\C\rangle$ where $A$ is an algebra  and $\C$ is a closure system on $A$. By the correspondence between closure systems and closure operators, we can also consider a g-matrix as a pair $\langle A,C\rangle$ where $C$ is a closure operator on $A$. 

The \emph{Tarski congruence} of a g-matrix $\langle A,C\rangle$, denoted by $\widetilde{\Omega}^AC$, is defined as the largest congruence $\theta$ on $A$ satisfying the following
\[
(a,b)\in\theta \implies C(a)=C(b).
\]

Given a logic $\S$, let us denote by $\Th(\S)$ the closure system of all theories of $\S$. Consider the Tarski congruence $\widetilde{\Omega}^{\Fm}\Th(\S)$ of the g-matrix $\langle\Fm,\Th(\S)\rangle$. The \emph{intrinsic variety} of $\S$, denoted by $\V(\S)$, is defined as the variety generated by the quotient algebra $\Fm/\widetilde{\Omega}^{\Fm}\Th(\S)$. That is, $\V(\S)=\V(\Fm/\widetilde{\Omega}^{\Fm}\Th(\S))$. Now, since the interderivability relation $\dashv\vdash_N$ is a congruence  on $\Fm$ (that is, the logic $\S_N$ is selfextensional \cite{Fo16}), we have that
\[
\V(\S_N)\models\alpha\approx\beta \iff \Fm/\widetilde{\Omega}^{\Fm}\Th(\S)\models\alpha\approx\beta \iff (\alpha,\beta)\in\widetilde{\Omega}^{\Fm}\Th(\S) \iff \alpha\dashv\vdash_N\beta.
\]

\begin{theorem}
$\V(\S_N)=\{A: A\models x\approx\neg\neg x\}$.
\end{theorem}

\begin{proof}
\onehalfspacing
Let $V:=\{A: A\models x\approx\neg\neg x\}$. Notice that if $\alpha=\neg^nx$, then for all $h\in\Hom(\Fm,A)$ with $A\in V$ , it follows that
\begin{equation}\label{equa: intrisinc variety}
\begin{cases}
h\alpha=h(\neg^nx)=hx			& \text{ if } n \text{ is even}\\
h\alpha=h(\neg^nx)=\neg hx		& \text{ if } n \text{ is odd}
\end{cases}
\end{equation}
Let $\alpha,\beta\in\Fm$. Let us prove that
\[
V\models\alpha\approx\beta \iff \exists x\in\var, \ \exists n,k\in\N_0 \text{ such that } \alpha=\neg^nx, \beta=\neg^kx \text{ and } n\equiv_2k.
\]

Suppose that $V\models\alpha\approx\beta$. By Lemma \ref{lem: abs free alg}, there are $x,y\in\var$ and $n,k\in\N_0$ such that $\alpha=\neg^nx$ and $\beta=\neg^ky$. If $x\neq y$, then taking $A=\{a,b\}$ with $\neg a=a$ and $\neg b=b$, we obtain that $A\in V$ and $A\not\models\alpha\approx\beta$. A contradiction. Hence $x=y$. Thus $\alpha=\neg^nx$ and $\beta=\neg^kx$. Suppose that $n\not\equiv_2k$. Thus $n$ is even and $k$ is odd (or vice versa). Taking $A_2=\{a,b\}$ as in page \pageref{alg A2} and $hx=a$, we obtain by \eqref{equa: intrisinc variety} that
\[
h\alpha=h(\neg^nx)=hx=a \quad \text{ and } \quad h\beta=h(\neg^kx)=\neg hx=\neg a=b.
\]
Hence $A\not\models\alpha\approx\beta$, which is a contradiction. Then $n\equiv_2k$.

Now assume that $\alpha=\neg^nx$, $\beta=\neg^kx$ and $n\equiv_2k$, for some $x\in\var$ and $n,k\in\N_0$. Let $A\in V$ and $h\in\Hom(\Fm,A)$. We have that $n$ and $k$ are both even or odd. In any case, by \eqref{equa: intrisinc variety} we obtain that $h\alpha=h\beta$. Hence $V\models\alpha\approx\beta$.

Hence, for all $\alpha,\beta\in\Fm$, we have
\begin{align*}
\V(\S_N)\models\alpha\approx\beta &\iff \alpha\dashv\vdash_N\beta\\ 
&\iff 
\exists x\in\var, \ \exists n,k\in\N_0 \ \text{ s.t. } \ \alpha=\neg^nx, \ \beta=\neg^kx \text{ and } n\equiv_2k\\
&\iff 
V\models \alpha\approx\beta.
\end{align*}
Therefore, $\V(\S_N)=\{A: A\models x\approx\neg\neg x\}$.
\end{proof}

\section{Full g-models for $\S_N$ and the class $\Alg(\S_N)$}\label{sec:6}

In this section, we obtain the class $\Alg(\S_N)$, which is considered in the framework of algebraic logic as \emph{the} algebraic counterpart of $\S_N$. To this end, we need first to obtain a characterization of the full g-models of $\S_N$. The class of full g-models of a logic $\S$ is a particular class of g-models of $\S$ that behave particularly well, in the sense that some interesting metalogical  properties of a propositional logic are precisely those shared by its full g-models. Then, the class $\Alg(\S)$ is obtained as the algebraic reducts of the reduced full g-models of $\S$. It is claimed (see \cite[p.~3]{FoJa09}) that the notion of full g-model is the ``right'' notion of model of a propositional logic and that the class $\Alg(\S)$ is the ``right'' class of algebras to be canonically associated with a propositional logic. We refer the reader to \cite{Fo16,FoJa09,FoJaPi03} for a more detailed argumentation about the full g-models of a logic $\S$ and the class $\Alg(\S)$.

We notice that the class $\Alg(\S_N)$ was already described in \cite{JaMo21} but using the notion of Suszko-reduced matrix models. As we mentioned before, we follow a different path to find the class $\Alg(\S_N)$.

We start presenting some needed notions. A g-matrix $\langle A,C\rangle$ is said to be a \emph{g-model} of a logic $\S$ when for all $\Gamma\cup\{\alpha\}\subseteq\Fm$,
\[
\Gamma\vdash_\S\alpha \implies \text{  for all } h\in\Hom(\Fm,A), \ h\alpha\in C(h\Gamma).
\]
Recall from the previous section the notion of Tarski congruence of a g-matrix. A g-matrix $\langle A,\C\rangle$ is said to be \emph{reduced} if $\widetilde{\Omega}^A\C=\mathrm{Id}_A$. Then, the class $\Alg(\S)$ is defined as follows:
\[
\Alg(\S)=\{A: \text{ there is a closure system } \C \text{ on } A \text{ such that } \langle A,\C\rangle \text{ is a reduced g-model of  } \S\}.
\]

An $h\in\Hom(A_1,A_2)$ is called a \emph{strict homomorphism} from the g-matrix $\langle A_1,C_1\rangle$ to the g-matrix $\langle A_2,C_2\rangle$ when
\[
a\in C_1(X) \iff ha\in C_2(hX) \qquad \text{for all } X\cup\{a\}\subseteq A_1.
\]

Let $A$ be an algebra. We denote by $\Fi_\S(A)$ the closure system of all $\S$-filters of the algebra $A$ of a logic $\S$. Given a subset $X\subseteq A$, let us denote by $\Fig_{\S}^A(X)$ the $\S$-filter of $A$ generated by $X$, that is, $\Fig_{\S}^A(X)$ is the least $\S$-filter of $A$ containing $X$.

A g-matrix $\langle A,\C\rangle$ is said to be a \emph{full g-model} of a logic $\S$ if there is an algebra $A_1$ and a surjective strict homomorphism $h\colon\langle A,\C\rangle\to\langle A_1,\Fi_\S(A_1)\rangle$. The above g-matrix $\langle A_1,\Fi_\S(A_1)\rangle$ can be chosen in a particular, significant way:

\begin{proposition}[{\cite[Prop. 5.85]{Fo16}}]\label{prop:full g-models}
A g-matrix $\langle A,\C\rangle$ is a full g-model of a logic $\S$ if and only if 
\[
\{F/\widetilde{\Omega}^A\C: F\in\C\}=\Fi_\S\left(A/\widetilde{\Omega}^A\C\right).
\]
\end{proposition}

\begin{proposition}[{\cite[Corollary 5.88(3)]{Fo16}}]
Given a logic $\S$,
\begin{multline*}
\Alg(\S)=\{A: \text{there is a closure system } \C \text{ on } A \text{ such that}\\
 \langle A,\C\rangle \text{ is a reduced full g-model of } \S\}.
\end{multline*}
\end{proposition}

In order to characterize the full g-models of $\S_N$, it will be useful to obtain a characterization of the $\S_N$-filters generated. We start with some basic properties.

\begin{proposition}\label{prop:properties of full basic g-models}
Let $A$ be an algebra. Then:
\begin{enumerate}[{\normalfont (B1)}]
	\item $\Fig_{\S_N}^A(\emptyset)=\emptyset$.
	\item $\Fig_{\S_N}^A(a,\neg a)=A$, for all $a\in A$.
	\item $\Fig_{\S_N}^A(a)=\Fig_{\S_N}^A(\neg\neg a)$.
\end{enumerate}
\end{proposition}

\begin{proof}
\onehalfspacing
(1)  is a consequence from the fact that the logic $\S_N$ have no theorems, and (2) and (3) are straightforward by rules (R1)--(R3).
\end{proof}

Now let us to obtain a characterization of $\Fig_{\S_N}^A(X)$. Let $\Ev=\{n\in\N_0: n \text{ is even}\}$ and $\Od=\{n\in\N_0: n \text{ is odd}\}$. The next is a key proposition.

\begin{proposition}\label{prop: characterization of generated filter}
Let $A$ be an algebra. Let $B\subseteq A$ be such that the following condition holds:
\begin{equation}\label{equa:condition for proper filter}
\forall b,b'\in B,\forall (s,t)\in\Ev\times\Od(\neg^sb\neq\neg^tb').
\end{equation}
Then,
\[
\Fig_{\S_N}^A(B)=\{a\in A: \exists b\in B, \exists s,t\in\Ev \, (\neg^sa=\neg^tb)\} \quad \text{and} \quad \Fig_{\S_N}^A(B)\neq A.
\]
\end{proposition}

\begin{proof}
\onehalfspacing
Let 
\[
F=\{a\in A: \exists b\in B, \exists s,t\in\Ev \, (\neg^sa=\neg^tb)\}.
\]
We prove in several steps that $\Fig_{\S_N}^A(B)=F$.

$\bullet$ $F\neq A$. If $B=\emptyset$, then $F=\emptyset\neq A$. Suppose that $B\neq\emptyset$, and let $b\in B$. We show that $\neg b\notin F$. Suppose that $\neg b\in F$. Thus, there is $b'\in B$ and $s,t\in \Ev$ such that $\neg^s(\neg b)=\neg^tb'$. Then $\neg^{s+1}b=\neg^tb'$ with $(t,s+1)\in\Ev\times\Od$, which is a contradiction by \eqref{equa:condition for proper filter}. Hence $\neg b\notin F$. Therefore, $F\neq A$.

$\bullet$ We show that $F$ is an $\S_N$-filter. We need to verify conditions (1) and (2) of Proposition \ref{prop: charact filter}. (1) Since $F\neq A$, we need to show that for all $a\in A$, $a\notin F$ or $\neg a\notin F$. Let $a\in A$. Suppose that $a,\neg a\in F$. Thus, there are $b,b'\in B$ and $s,t,s',t'\in\Ev$ such that $\neg^s a=\neg^tb$ and $\neg^{s'}(\neg a)=\neg^{t'}b'$.  Then $\neg^{s+s'+1}a=\neg^{t+s'+1}b$ and $\neg^{s+s'+1}a=\neg^{s+t'}b'$. Hence $\neg^{t+s'+1}b=\neg^{s+t'}b'$ with $b,b' \in B$ and $(s+t',t+s'+1)\in\Ev\times\Od$, which is a contradiction by \eqref{equa:condition for proper filter}. Thus $a\notin F$ or $\neg a\notin F$. (2) Let $a\in A$. If $a\in F$, then there is $b\in B$ and $s,t\in \Ev$ such that $\neg^sa=\neg^tb$. Hence $\neg^s(\neg\neg a)=\neg^{t+2}b$. By definition of $F$, we obtain that $\neg\neg a\in F$. Conversely, suppose that $\neg\neg a\in F$. Thus, there is $b\in B$ and $s,t\in\Ev$ such that $\neg^s(\neg\neg a)=\neg^tb$. Hence $\neg^{s+2}a=\neg^tb$ and $s+2,t\in\Ev$. Then $a\in F$. Therefore, it follows by Proposition \ref{prop: charact filter} that $F$ is an $\S_N$-filter.

$\bullet$ $B\subseteq F$. It is obvious by definition of $F$.

$\bullet$ $F$ is the least $\S_N$-filter of $A$ containing $B$. Let $G$ be an $\S_N$-filter of $A$ such that $B\subseteq G$. Let $a\in F$. Thus, there is $b\in B$ and $s,t\in\Ev$ such that $\neg^sa=\neg^tb$. By Proposition \ref{prop:property of vdash_N}, we have that $x\vdash_N\neg^tx$. Since $b\in B\subseteq G$ and $G$ is an $\S_N$-filter, it follows that $\neg^tb\in G$. Thus $\neg^sa\in G$. Since $\neg^sx\vdash_Nx$, it follows that $a\in G$. Hence $F\subseteq G$. Therefore, we have proved that $\Fig_{\S_N}^A(B)=F$.
\end{proof}

\begin{proposition}\label{prop:characterization of proper generated filter}
Let $A$ be an algebra and  $B\subseteq A$. The following conditions are equivalent.
\begin{enumerate}[(1)]
	\item There exist $b,b'\in B$ and $(s,t)\in\Ev\times\Od$ such that $\neg^sb=\neg^tb'$.
	\item $\Fig_{\S_N}^A(B)=A$.
\end{enumerate}
\end{proposition}

\begin{proof}
\onehalfspacing
$(1)\Rightarrow(2)$ Let $b,b'\in B$ and $(s,t)\in\Ev\times\Od$ such that $\neg^sb=\neg^tb'$. We have that $b,b'\in \Fig_{\S_N}^A(B)$. Since $s$ is even, it follows that $x\vdash_N\neg^sx$. Then $\neg^sb\in\Fig_{\S_N}^A(B)$. Thus $\neg^tb'\in\Fig_{\S_N}^A(B)$. Since $t$ is odd, it follows by Proposition \ref{prop:property of vdash_N} that $\neg^tx\vdash_N\neg x$. Then $\neg b'\in \Fig_{\S_N}^A(B)$. That is, $b',\neg b'\in\Fig_{\S_N}^A(B)$. Hence $\Fig_{\S_N}^A(B)=A$. 

$(2)\Rightarrow(1)$ It follows by Proposition \ref{prop: characterization of generated filter}.
\end{proof}

From Propositions \ref{prop: characterization of generated filter} and \ref{prop:characterization of proper generated filter}, we have that for every algebra $A$ and every $B\subseteq A$,
\[
\Fig_{\S_N}^A(B)=A \quad \text{or} \quad \Fig_{\S_N}^A(B)=\{a\in A: \exists b\in B, \exists s,t\in\Ev \, (\neg^sa=\neg^tb)\}.
\]

\begin{proposition}\label{prop:properties of Fig}
Let $A$ be an algebra and $b\in A$ such that
\begin{equation}\label{equa:property generated filter of one element}
\forall (s,t)\in\Ev\times\Od \, (\neg^sb\neq\neg^tb).
\end{equation}
Then,
\begin{enumerate}
	\item[{\normalfont (1)}] $\Fig_{\S_N}^A(b)=\{a\in A: \exists s,t\in\Ev \, (\neg^sa=\neg^tb)\}$;
	\item[{\normalfont (B4)}] $a\in\Fig_{\S_N}^A(b) \implies b\in\Fig_{\S_N}^A(a)$;
	\item[{\normalfont (B5)}] $a\in\Fig_{\S_N}^A(b) \implies \neg b\in\Fig_{\S_N}^A(\neg a)$.
\end{enumerate}
\end{proposition}

\begin{proof}
\onehalfspacing
(1) is an immediate consequence of Proposition \ref{prop: characterization of generated filter}. (B4) and (B5) follow by (1) and from Proposition \ref{prop:property of vdash_N}.
\end{proof}

Let $A$ be an algebra. Notice from Proposition \ref{prop:characterization of proper generated filter} that an element $b\in A$ satisfies condition  \eqref{equa:property generated filter of one element} if and only if $\Fig_{\S_N}^A(b)\neq A$. Moreover if $B\subseteq A$ satisfying condition \eqref{equa:condition for proper filter}, then every element $b\in B$ satisfies condition \eqref{equa:property generated filter of one element}. Thus, the next proposition is a consequence of Proposition \ref{prop: characterization of generated filter} and by (1) of Proposition \ref{prop:properties of Fig}.

\begin{proposition}\label{prop:property of Fig 2}
Let $A$ be an algebra and $B\subseteq A$.  Then,
\begin{enumerate}
	\item[{\normalfont (B6)}]  if $a\in \Fig_{\S_N}^A(B)$ and $\Fig_{\S_N}^A(B)\neq A$, then there is $b\in B$ such that $a\in\Fig_{\S_N}^A(b)$.
\end{enumerate}
\end{proposition}

\begin{proposition}
Let $\langle A,C\rangle$ be a g-matrix. Then
\[
(a,b)\in\widetilde{\Omega}^AC \iff \forall n\in \N_0(C(\neg^na)=C(\neg^nb)).
\]
\end{proposition}

\begin{proof}
\onehalfspacing
It follows from the fact that $\widetilde{\Omega}^AC=\bigcap_{F\in\C}\Omega^AF$ (see \cite[Lem. 5.31]{Fo16}) and by Proposition \ref{prop:Leibniz cong}.
\end{proof}

\begin{proposition}\label{prop:Tarski cong 2}
Let $\langle A,C\rangle$ be a g-matrix such that for each $a\in A$, $C(a)=C(\neg\neg a)$. Then, for all $a,b\in A$,
\[
(a,b)\in\widetilde{\Omega}^AC \iff C(a)=C(b) \text{ and } C(\neg a)=C(\neg b).
\]
\end{proposition}

\begin{proof}
\onehalfspacing
The implication $\Rightarrow$ follows by the previous proposition. 

$(\Leftarrow)$ Let $a,b\in A$ be such that $C(a)=C(b)$ and $C(\neg a)=C(\neg b)$. We prove that $\forall n\in \N_0(C(\neg^na)=C(\neg^nb))$ by induction on $n$. It is clear that holds for $n=0$ and $n=1$. Let $n\geq 2$ and suppose that for all $k<n$, $C(\neg^ka)=C(\neg^kb)$. Since $n\geq 2$, it follows that
\[
C(\neg^na)=C(\neg^{n-2}a)\stackrel{I.H.}{=}C(\neg^{n-2}b)=C(\neg^nb).
\]
Hence, $\forall n\in \N_0(C(\neg^na)=C(\neg^nb))$. Therefore $(a,b)\in\widetilde{\Omega}^AC$.
\end{proof}

Now we are ready to present a characterization for the full g-models of the negation fragment $\S_N$.

\begin{theorem}
Let $\langle A,C\rangle$ be a g-matrix. Then, $\langle A,C\rangle$ is a full g-model of the logic $\S_N$ if and only if it satisfies:
\begin{enumerate}[{\normalfont (F1)}]
	\item $C(a,\neg a)=A$, for all $a\in A$.
	\item $C(a)=C(\neg\neg a)$, for all $a\in A$.
	\item For all $a,b\in A$, if $a\in C(b)$ and $C(b)\neq A$, then $b\in C(a)$.
	\item For all $a,b\in A$, if $a\in C(b)$ and $C(b)\neq A$, $\neg b\in C(\neg a)$.
	\item For all $B\cup\{a\}\subseteq A$, if $a\in C(B)$ and $C(B)\neq A$, then there is $b\in B$ such that $a\in C(b)$.
	\item For every $B\subseteq A$, if $C(B)=A$, then there are $b,b'\in B$ such that $C(b)=C(\neg b')$.
\end{enumerate}
\end{theorem}

\begin{proof}
\onehalfspacing
Let $\langle A, C\rangle$ be a g-matrix and let $\C$ be the closure system associated with $C$.

$(\Rightarrow)$ Assume that $\langle A,C\rangle$ is a full g-model of the logic $\S_N$. Thus, there is an algebra $A_1$ and a surjective strict homomorphism $h\colon\langle A,\C\rangle\to\langle A_1,\Fi_{\S_N}(A_1)\rangle$. Conditions (F1) and (F2) follow by Proposition \ref{prop:properties of full basic g-models} and from the fact that theses conditions are preserved by surjective strict homomorphisms (see \cite[Proposition 5.90]{Fo16}).

\noindent (F3) Let $a\in C(b)$ and suppose that $C(b)\neq A$. Since $h$ is a strict homomorphism, it follows that $ha\in \Fig_{\S_N}^{A_1}(hb)$ and $\Fig_{\S_N}^{A_1}(hb)\neq A$. Thus, by Proposition \ref{prop:properties of Fig}, we have that $hb\in\Fig_{\S_N}^{A_1}(ha)$. Then $b\in C(a)$.

\noindent (F4) It is similar to the proof of (F3).

\noindent (F5) Let $a\in C(B)$ and suppose that $C(B)\neq A$. Then, we have that $ha\in\Fig_{\S_N}^{A_1}(hB)$ and $\Fig_{\S_N}^{A_1}(hB)\neq A$. By Proposition \ref{prop:property of Fig 2}, there is $b\in B$ such that $ha\in\Fig_{\S_N}^{A_1}(hb)$. Then $a\in C(b)$ with $b\in B$.

\noindent (F6) Let $B\subseteq A$ be such that $C(B)=A$. Since $h$ is a surjective strict homomorphism, it follows that $\Fig_{\S_N}^{A_1}(hB)=A_1$. By Proposition \ref{prop:characterization of proper generated filter}, there are $b,b'\in B$ and $(s,t)\in\Ev\times\Od$ such that $\neg^shb=\neg^thb'$. Then, $\Fig_{\S_N}^{A_1}(h(\neg^s b))=\Fig_{\S_N}^{A_1}(h(\neg^t b'))$. Thus $C(\neg^s b)=C(\neg^t b')$. By (F2), it follows that $C(b)=C(\neg^sb)=C(\neg^tb')=C(\neg b')$.

$(\Leftarrow)$ Assume that $\langle A,C\rangle$ is a g-matrix satisfying conditions (F1)--(F6). Let $\langle A^*,C^*\rangle$ be the reduction of $\langle A,C\rangle$. That is, $A^*=A/\widetilde{\Omega}^AC$ and $\C^*=\{F/\widetilde{\Omega}^AC: F\in\C\}$. By Proposition \ref{prop:full g-models}, it is enough to prove that $\C^*=\Fi_{\S_N}(A^*)$. Notice that the natural homomorphism $\pi\colon A\to A^*$ is a surjective strict homomorphism from the g-matrix $\langle A,C\rangle$ onto its reduction $\langle A^*,C^*\rangle$. For each $a\in A$, let $\overline{a}=a/\widetilde{\Omega}^A\C$.

$(\subseteq)$ By (F1) and (F2), we have that the g-matrix $\langle A,C\rangle$ is a g-model of $\S_N$. Then, $\langle A^*,C^*\rangle$ is a g-model of $\S_N$. Hence $\C^*\subseteq\Fi_{\S_N}(A^*)$. 

$(\supseteq)$ In order to prove the inverse inclusion, let us show that $C^*(B^*)\subseteq\Fig_{\S_N}^{A^*}(B^*)$, for all $B\subseteq A$ where $B^*=\{\overline{b}: b\in B\}$. Let $B\subseteq A$. If $\Fig_{\S_N}^{A^*}(B^*)=A^*$, then $C^*(B^*)\subseteq\Fig_{\S_N}^{A^*}(B^*)$. Assume that $\Fig_{\S_N}^{A^*}(B^*)\neq A^*$. Let us show that $C^*(B^*)\neq A^*$. Suppose by contradiction that $C^*(B^*)=A^*$. Thus $C(B)=A$. By (F6), there are $a,a'\in B$ such that $C(a)=C(\neg a')$. Now we show that $C(a)\neq A$. Suppose that $C(a)=A$. Thus, by (F6), we obtain that $C(a)=C(\neg a)$. By (F2), we have that $C(\neg a)=C(\neg(\neg a))$. Then, by Proposition \ref{prop:Tarski cong 2}, it follows that $\overline{a}=\overline{\neg a}=\neg\overline{a}$. Since $\overline{a}\in B^*$, we have that $\overline{a},\neg\overline{a}\in\Fig_{\S_N}^{A^*}(B^*)$. Then $\Fig_{\S_N}^{A^*}(B^*)=A^*$, which is a contradiction. Hence $C(\neg a')=C(a)\neq A$. It follows by (F4) and (F2) that $C(\neg a)=C(a')$. Then $\overline{a}=\neg\overline{a'}$. Given that $\overline{a},\overline{a'}\in B^*$, we obtain that $\Fig_{\S_N}^{A^*}(B^*)=A^*$. A contradiction. Hence $C^*(B^*)\neq A^*$. And thus $C(B)\neq A$. Now, let $\overline{a}\in C^*(B^*)$. So $a\in C(B)$. By (F5), there is $a_0\in B$ such that $a\in C(a_0)$. Since $C(B)\neq A$, it follows that $C(a_0)\neq A$. Thus, since $a\in C(a_0)\neq A$, it follows by (F3) and (F4) that $a_0\in C(a)$ and $\neg a_0\in C(\neg a)$. Thus $C(a)=C(a_0)$. Since $C(a)=C(a_0)\neq A$ and $a_0\in C(a)$, we have by (F4) that $\neg a\in C(\neg a_0)$. Then, we obtain that $C(a)=C(a_0)$ and $C(\neg a)=C(\neg a_0)$. Thus $\overline{a}=\overline{a_0}\in B^*$. Hence $\overline{a}\in\Fig_{\S_N}^{A^*}(B^*)$. Hence, $C^*(B^*)\subseteq \Fig_{\S_N}^{A^*}(B^*)$.

We have proved that $\C^*=\Fi_{\S_N}(A^*)$. Therefore, $\langle A,C\rangle$ is a full g-model of $\S_N$.
\end{proof}

\begin{remark}
Notice that for every full g-model $\langle A,C\rangle$, it follows that $C(\emptyset)=\emptyset$. Indeed, condition (F6) implies that $C(\emptyset)\neq A$. Then condition (F5) implies that $a\notin C(\emptyset)$ for all $a\in A$. Hence $C(\emptyset)=\emptyset$.
\end{remark}

Notice that conditions $C(\emptyset)=\emptyset$ and (F1)--(F5) coincide respectively with conditions (B1)--(B6) when they are considered on the g-matrices $\langle A,\Fi_{\S_N}(A)\rangle$. And condition (F6) coincides with the implication (2)$\Rightarrow$(1) of Proposition \ref{prop:characterization of proper generated filter}.

Let $\langle A,C\rangle$ be a g-matrix. The \emph{Frege relation}, denoted by $\Lambda_AC$, of $C$ on $A$ is defined as follows:
\[
(a,b)\in\Lambda_AC \iff C(a)=C(b)
\]
for all $a,b\in A$. Notice that the Tarski congruence is the largest congruence below $\Lambda_AC$.

\begin{proposition}
Let $\langle A,C\rangle$ be a g-matrix satisfying conditions (F4) and (F6). Then, the Frege relation $\Lambda_AC$ is a congruence on $A$.
\end{proposition}

\begin{proof}
\onehalfspacing
Let $a,b\in A$ be such that $(a,b)\in\Lambda_AC$. Thus $C(a)=C(b)$. If $C(a)=C(b)\neq A$, then we have $a\in C(b)\neq A$. Then, we obtain by (F4) that $\neg b\in C(\neg a)$. Analogously, $\neg a\in C(\neg b)$. Hence $C(\neg a)=C(\neg b)$. On the other hand, if $C(a)=C(b)=A$, it follows by (F6) that $C(a)=C(\neg a)$ and $C(b)=C(\neg b)$. Then $C(\neg a)=C(\neg b)$. Hence $(\neg a,\neg b)\in\Lambda_AC$.
\end{proof}

The following proposition tells us that for an algebra $A$ in the intrinsic variety of $\S_N$, the proper generated $\S_N$-filters on $A$ are quite simple.

\begin{proposition}\label{prop:generated S-filters}
Let $A\in\V(\S_N)$. For all $B\subseteq A$, if $\Fig_{\S_N}^A(B)\neq A$, then $\Fig_{\S_N}^A(B)=B$.
\end{proposition}

\begin{proof}
\onehalfspacing
Let $A\in\V(\S_N)$. Suppose that $\Fig_{\S_N}^A(B)\neq A$ and let $a\in\Fig_{\S_N}^A(B)$. By Proposition \ref{prop: characterization of generated filter}, there is $b\in B$ and $s,t\in\Ev$ such that $\neg^n a=\neg^tb$. Since $A\models x\approx\neg\neg x$, it follows that $a=b\in B$. Hence $\Fig_{\S_N}^A(B)=B$.
\end{proof}

\begin{theorem}\label{theo:Alg(S)}
\[
\Alg(\S_N)=\{A: A\models x\approx\neg\neg x \ \text{ and } \ A\models(x\approx\neg x \ \& \ y\approx\neg y \implies x\approx y)\}.
\]
\end{theorem}

\begin{proof}
\onehalfspacing
$(\subseteq)$ Let $A\in\Alg(\S_N)$. Thus $\langle A,\Fi_{\S_N}(A)\rangle$ is a reduced full g-model of $\S_N$. Then $\langle A,\Fi_{\S_N}(A)\rangle$ satisfies conditions (F1)--(F6). By the previous proposition, it follows that $\Lambda_A\Fi_{\S_N}(A)=\widetilde{\Omega}^A\Fi_{\S_N}(A)=\mathrm{Id}_A$. Then, by (B3) (or (F2)), we obtain that $a=\neg\neg a$, for all $a\in A$. Hence $A\models x\approx\neg\neg x$. Now let $a,b\in A$ and assume that $a=\neg a$ and $b=\neg b$. By (B2), we have $\Fig_{\S_N}^A(a)=A=\Fig_{\S_N}^A(b)$. Thus $(a,b)\in\Lambda_A\Fi_{\S_N}(A)$. Then $a=b$. Hence $A\models(x\approx\neg x \ \& \ y\approx\neg y\implies x\approx y)$.

$(\supseteq)$ Let $A$ be an algebra such that $A\models x\approx\neg\neg x$ and $A\models(x\approx\neg x \ \& \ y\approx\neg y\implies x\approx y)$.  It is straightforward that the g-matrix $\langle A,\Fi_{\S_N}(A)\rangle$ is a full g-model of $\S_N$. Let us show that the Frege relation of $\langle A,\Fi_{\S_N}(A)\rangle$ is the identity relation. Let $(a,b)\in\Lambda_A\Fi_{\S_N}(A)$. Thus $\Fig_{\S_N}^A(a)=\Fig_{\S_N}^A(b)$. If $\Fig_{\S_N}^A(a)=\Fig_{\S_N}^A(b)\neq A$, it follows by Proposition \ref{prop:generated S-filters} that $a=b$. Assume that $\Fig_{\S_N}^A(a)=\Fig_{\S_N}^A(b)=A$. From $\Fig_{\S_N}^A(a)=A$, it follows by Proposition \ref{prop:characterization of proper generated filter} that there is $(s,t)\in\Ev\times\Od$ such that $\neg^sa=\neg^ta$. Then $a=\neg a$. Similarly, $b=\neg b$. Since $A\models(x\approx\neg x \ \& \ y\approx\neg y\implies x\approx y)$, we obtain that $a=b$. Hence, $\Lambda_A\Fi_{\S_N}(A)=\mathrm{Id}_A$. Therefore, the full g-model $\langle A,\Fi_{\S_N}(A)\rangle$ is reduced, and thus $A\in\Alg(\S_N)$.
\end{proof}

\section{Concluding remarks}\label{sec:7}

Throughout this article we have obtained the three classes of algebras that are canonically associated with the logic $\S_N$ in the context of algebraic logic:
\[
\Alg^*(\S_N)=\mathbb{I}(A_1,A_2,A_3) \qquad  \qquad \V(\S)=\{A: A\models x\approx\neg\neg x\}
\]
\[
\Alg(\S_N)=\{A: A\models (x\approx\neg\neg x) \ \& \ (x\approx\neg x \ \& \ y\approx\neg y\implies x\approx y)\}.
\]
With this, we can complete the table with all the fragments of CPL and their corresponding classes of algebras canonically associated, see Table \ref{tab:fragments}. 

\begin{table}[]
\centering
\resizebox{\textwidth}{!}{%
\begin{tabular}{|c@{\vrule height 25pt depth 20pt width 0pt}|c|c|c|c|}
\hline
\multicolumn{1}{|c|}{\begin{tabular}[c]{@{}l@{}} \hspace{25mm} Classes\\ Fragments\end{tabular}} &
  $\Alg^*(\S_N)$ &
  $\Alg(\S_N)$ &
  $\V(\S_N)$ &
  For instance, see \\ \hline
$\{\neg\}$ &
  $\mathbb{I}(A_1,A_2,A_3)$ &
  \begin{tabular}[c]{@{}c@{}}$A\models x\approx\neg\neg x$ and\\ $A\models(x\approx\neg x \ \& \ y\approx \neg y\implies x\approx y)$\end{tabular} &
  $A\models x\approx\neg\neg x$ &
   \\ \hline
$\{\wedge\}$ &
  $\mathbb{I}({\bf 1}_\wedge, {\bf 2}_\wedge)$ &
  (meet)-Semilattices &
  (meet)-Semilattices &
  \begin{tabular}[c]{@{}c@{}}\cite{FoMo14}\\ (also \cite[pp.~208--209]{Fo16})\end{tabular} \\ \hline
$\{\vee\}$ &
  \begin{tabular}[c]{@{}c@{}}(Join)-Semilattices with 1 + \\ $a<b\implies\exists c(a\vee c\neq 1 \ \& \ b\vee c=1)$\end{tabular} &
  (join)-Semilattices &
  (join)-Semilattices &
  \cite{FoMo14} \\ \hline
$\{\to\}$ or $\{\vee,\to\}$ &
  \begin{tabular}[c]{@{}c@{}}Implication algebras \\ (or Tarski algebras)\end{tabular} &
  \begin{tabular}[c]{@{}c@{}}Implication algebras\\ (or Tarski algebras)\end{tabular} &
  \begin{tabular}[c]{@{}c@{}}Implication algebras\\ (or Tarski algebras)\end{tabular} &
  \begin{tabular}[c]{@{}c@{}}\cite{Ra74}\\ (also \cite[p.~85]{Fo16})\end{tabular} \\ \hline
$\{\wedge,\vee\}$ &
  \begin{tabular}[c]{@{}c@{}}Distributive lattices with 1 +\\ $a<b\implies\exists c(a\vee c\neq 1 \ \& \ b\vee c=1)$\end{tabular} &
  Distributive lattices &
  Distributive lattices &
  \cite{FoGuVe91,FoVe91} \\ \hline
$\{\wedge,\to\}$ or $\{\wedge,\vee,\to\}$ &
  \begin{tabular}[c]{@{}c@{}}Relatively complemented distributive\\ lattices with upper bound 1\end{tabular} &
  \begin{tabular}[c]{@{}c@{}}Relatively complemented distributive\\ lattices with upper bound 1\end{tabular} &
  \begin{tabular}[c]{@{}c@{}}Relatively complemented distributive\\ lattices with upper bound 1\end{tabular} &
   \\ \hline
$\{\wedge,\neg\}$, $\{\vee,\neg\}$ or $\{\to,\neg\}$ &
  Boolean algebras &
  Boolean algebras &
  Boolean algebras &
  \cite{Fo16} \\ \hline
  $\{\leftrightarrow\}$ &
  Boolean groups &
  Boolean groups &
  Boolean groups &
  \begin{tabular}[c]{@{}c@{}}\cite[p.~57]{BloPi89}\\ (also \cite[p.~131]{Fo16})\end{tabular} \\ \hline
\end{tabular}%
}
\caption{The fragments of CPL and their classes of algebras.}
\label{tab:fragments}
\end{table}

Notice that it is clear that the inclusions $\Alg^*(\S_N)\subset\Alg(\S_N)\subset\V(\S_N)$ are proper. For instance, the algebra $A_2=\{a,b\}$, with $\neg a=a$ and $\neg b=b$, belongs to  $\V(\S_N)$ but not to $\Alg(\S_N)$. The $\{\neg\}$-reduct of the four-element Boolean algebra belongs to $\Alg(\S_N)$ but not to $\Alg^*(\S_N)$. Also notice that the class $\Alg(\S_N)$ is a quasi-variety but not a variety. 

From the results that we have obtained, we can classify the $\{\neg\}$-fragment of CPL in the Leibniz hierarchy and in the Frege hierarchy (the two hierarchies of algebraic logic, see \cite[Chap. 6 and 7]{Fo16}) and respond to an open problem proposed in \cite[p.~418]{Fo16}. The $\{\neg\}$-fragment $\S_N$ of CPL is outside of the Leibniz hierarchy because it is neither protoalgebraic nor truth-equational: Notice that the logic $\S_N$ is non-trivial, and thus $\S_N$ is not almost inconsistent. Since the unique protoalgebraic logic without theorems is the almost inconsistent one, it follows that $\S_N$ is not protoalgebraic. Moreover since the logic $\S_N$ has no theorems, it follows that $\S_N$ is not truth-equational. Now we turn out our attention to the Frege hierarchy (see \cite[p.~414]{Fo16}). Since being Fregean is a property that is  preserved by fragments and CPL is clearly Fregean, it follows that $\S_N$ is Fregean. Moreover, the logic $\S_N$ is also fully selfextensional: Let $A\in\Alg(\S_N)$. In the proof of Theorem \ref{theo:Alg(S)} we have proved that the Frege relation $\Lambda_A\Fi_{\S_N}(A)$ is the identity relation. Hence $\S_N$ is fully selfextensional. However, we will show that $\S_N$ is not fully Fregean. In order to show this, consider the following relation: let $\S$ be a logic, $A$  an algebra and $F\subseteq A$,
\[
\Lambda_{\S}^AF=\{(a,b)\in A\times A: \Fig_\S^A(F,a)=\Fig_\S^A(F,b)\}.
\]
\begin{proposition}[{\cite[Prop. 7.56]{Fo16}}]
A logic $\S$ is fully Fregean if and only if $\Lambda_\S^AF\subseteq\Omega^AF$ for every $F\in\Fi_\S(A)$ and every algebra $A$.
\end{proposition}
Let $A_3=\{a,b,c\}$ be the algebra given on page \pageref{alg A3}. Let $F=\{a\}$. We know that $F\in\Fi_{\S_N}(A_3)$. We have that $\Fig_{\S_N}^{A_3}(F,b)=A$ (because $b,\neg b\in\Fig_{\S_N}^{A_3}(F,b)$) and $\Fig_{\S_N}^{A_3}(F,c)=A$ (because $c,\neg c\in\Fig_{\S_N}^{A_3}(F,c)$). Hence $(b,c)\in\Lambda_{\S_N}^AF$. But $(b,c)\notin\Omega^AF$ because $\neg b\in F$ and $\neg c\notin F$ (see Proposition \ref{prop:Leibniz cong}). Thus $\Lambda_{\S_N}^AF\nsubseteq\Omega^AF$.  Therefore, the $\{\neg\}$-fragment of CPL is not fully Fregean. This answers negatively the question: Is the class of fully Fregean logics the intersection of the classes of the Fregean and the fully selfextensional ones? That is, $\S_N$ is a Fregean and fully selfextensional logic but is not fully Fregean one. It is worth mentioning that Tommaso Moraschini and Ramon Jansana found this example before but they don't publish it (c.p.).

\section*{Funding}

This work was supported by the Fondo para la Investigación Científica y Tecnológica (Argentina) [grant number PICT-2019-0674, PICT-2019-0882]; and the  Universidad Nacional de La Pampa [grant number P.I. 78M].

\section*{Acknowledgements}

We are grateful to the referee for carefully reading the manuscript and making useful suggestions that have improved the arguments' foundation. We are also thankful to the editors for their valuable suggestions.

\end{document}